\documentclass[centertags,12pt]{amsart}

\usepackage{amssymb}
\usepackage{latexsym}
\usepackage{amsfonts}
\usepackage{amssymb}
\usepackage{graphicx} 
\usepackage{pdflscape} 
\usepackage{diagbox} 

\numberwithin{equation}{section}

\makeatletter
\renewcommand{\subsection}{\@startsection
{subsection}{2}{0mm}{\baselineskip}{-0.25cm}
{\normalfont\normalsize\em}}
\makeatother

\newtheorem{theorem}{Theorem}[section]
\newtheorem{proposition}[theorem]{Proposition}
\newtheorem{corollary}[theorem]{Corollary}
\newtheorem{lemma}[theorem]{Lemma}

{\theoremstyle{definition}
\newtheorem{definition}[theorem]{Definition}
\newtheorem{example}[theorem]{Example}
\newtheorem{Ex2.1}[theorem]{Example 2.5 revisited}
\newtheorem{Ex2.2}[theorem]{Example 2.6 revisited}

}

{\theoremstyle{remark}
\newtheorem{remark}[theorem]{Remark}
\newtheorem{scholium}[theorem]{Scholium}
}

\def\1{\mathbf 1}
\def\e{\mathbf e}

\def\x{\mathbf x}

\def\F{\mathbf F}

\def\T{\mathbf T}
\def\R{\mathbf R}

\def\bN{\mathbb N}
\def\bZ{\mathbb Z}

\def\cO{\mathcal O}

\def\cL{\mathcal L}

\def\cO{\mathcal O}

\def\cS{\mathcal S}

\def\N {\mathbb{N}}

\def\e {\epsilon}
\def\f {\varphi}
\def\g {\gamma}

\def\o {\omega}

\def\L {\left}
\def\R {\right}

\title[Numerical semigroups by genus and even gaps]{Counting 
numerical semigroups \\ by genus and even gaps}

\author[M. Bernardini]{Matheus Bernardini}
\address{IMECC/UNICAMP, R. S\'ergio Buarque de Holanda 651, Cidade 
Universit\'aria \lq\lq Zeferino Vaz", 13083-859, Campinas, 
SP-Brazil}
\address{Instituto Federal de S\~ao Paulo, Campinas, SP-Brazil}
\email{matheus.bernardini@gmail.com}

\author[F. Torres]{Fernando Torres}
\address{IMECC/UNICAMP, R. S\'ergio Buarque de Holanda 651, Cidade 
Universit\'aria \lq\lq Zeferino Vaz", 13083-859, Campinas, 
SP-Brazil}
\email{ftorres@ime.unicamp.br}

\thanks{{\em 2010 Math. Subj. Class.}: Primary 20M14; 
Secondary 05A15, 05A19}

\thanks{{\em Keywords}: numerical semigroup, even gap, genus, 
$\gamma$-hyperelliptic semigroup, $f_\gamma$ sequence}

\begin{document}


   \begin{abstract} Let $n_g$ be the number of numerical semigroups of genus $g$. We present 
   an approach to compute $n_g$ by using even gaps, and the question: Is it true that 
   $n_{g+1}>n_g$? is investigated. Let $N_\gamma(g)$ be the number of numerical semigroups 
   of genus $g$ whose number of even gaps equals $\gamma$. We show that 
   $N_\gamma(g)=N_\gamma(3\gamma)$ for $\gamma\leq \lfloor g/3\rfloor$ and $N_\gamma(g)=0$ 
   for $\gamma>\lfloor 2g/3\rfloor$; thus the question above is true provided that 
   $N_\gamma(g+1)>N_\gamma(g)$ for $\gamma=\lfloor g/3\rfloor +1,\ldots,\lfloor 
   2g/3\rfloor$. We also show that $N_\gamma(3\gamma)$ coincides with $f_\gamma$, the number 
   introduced by Bras-Amor\'os \cite{Amoros3} in conection with semigroup-closed sets. 
   Finally, the stronger possibility $f_\gamma\sim \varphi^{2\gamma}$ arises being 
   $\varphi=(1+\sqrt{5})/2$ the golden number. \end{abstract}
    \maketitle

   \section{Introduction}\label{s1}

A {\em numerical semigroup} $S$ is a submonoid of the set of 
nonnegative integers $\bN_0$, equipped with the usual addition, such 
that $G(S):=\bN_0\setminus S$, the set of {\em gaps} of $S$, is 
finite. The number of elements $g=g(S)$ of $G(S)$ is called the {\em 
genus} of $S$ and thus the semigroup property implies (see e.g. 
\cite[Lemma 2.14]{GS-R})
   \begin{equation}\label{eq1.1}
   S\supseteq \{2g+i:i\in\bN_0\}\, .
   \end{equation}
Suitable references for the background on numerical semigroups that we assume are in fact 
the books \cite{GS-R} and \cite{RA}. In spite of its simplicity, as a mathematical object, a 
numerical semigroup often plays a key role in the study of more involved or subtle 
structures arising e.g. in Algebraic Curve Theory \cite{Kato}, \cite{Arnaldo}, 
\cite{Torres1}, \cite{Oliveira-Pimentel}, \cite{Komeda} or e.g. in Coding Theory 
\cite{Pellikaan-Torres}, \cite{Amoros4}.

In this paper we deal with a problem of purely combinatorial nature, namely: For $g\in 
\bN_0$ given, find the number $n_g$ of elements of the family $\cS_g$ of numerical 
semigroups of genus $g$; Kaplan \cite{Kaplan2} wrote a nice survey and state of the art on 
this problem, and one can find information on these numbers in Sloane's On-line Encyclopedia 
of Integer Sequences \cite{Sloane}. Indeed, our goal here is the question (\ref{eq1.2}) 
below.

We have $n_g\leq \binom{2g-1}{g}$ by (\ref{eq1.1}) and in fact, a better bound is known, 
namely $n_g\leq\frac{1}{g+1}\binom{2g}{g}$ which was obtained by Bras-Amor\'os and de Mier 
via so-called Dyck paths \cite{BdM}. Further bounds on $n_g$ were computed by Bras-Amor\'os 
\cite{Amoros2} via the semigroup tree method; see also Bras-Amor\'os and Bulygin 
\cite{Amoros-Bulygin}, O'Dorney \cite{Dorney}, Elizalde \cite{Elizalde}. Blanco and Rosales 
\cite{BR} approached this problem by considering a partition of $\cS_g$ by subsets of 
semigroups $S$ of a given Frobenius number $F=F(S)$, which by definition is the biggest 
integer which does not belong to $S$; see also \cite{BG-SP}. In any case, computing the 
exact value of $n_g$ seems to be out of reach although there exist algorithmic methods for 
determining such a number \cite{FH}, \cite{BA-F}.

By taking into consideration the first $50$ values of $n_g$, 
Bras-Amor\'os \cite{Amoros1} conjectured Fibonacci-like properties 
on the behaviour of the sequence $n_g$:
   \begin{enumerate}
\item[\rm(A)] $n_{g+2}\geq n_{g+1}+n_g$ for any $g$;
\item[\rm(B)] $\lim_{g\to\infty}\frac{n_{g+1}+n_g}{n_{g+2}}=1$;
\item[\rm(C)] $\lim_{g\to\infty}\frac{n_{g+1}}{n_g}=\varphi:= 
\frac{1+\sqrt{5}}{2}$, so-called {\em golden number}.
   \end{enumerate}
Indeed, Conjectures (B) and (C) have been recently proved by Zhai 
\cite{Zhai}. Here we focus in the following problem suggested by (A) 
whose answer is positive for large $g$ by (C) or $g\leq 50$ by the 
aforementioned values in \cite{Amoros1} (which were recently extended to $g\leq 67$ in \cite{FH}):
   \begin{equation}\label{eq1.2}
   \text{Is it true that $n_{g+1}>n_g$ for any $g\geq 1$?}
   \end{equation}
   The {\em multiplicity} $m(S)$ of a numerical semigroup $S$ is its first 
   positive element. Kaplan \cite{Kaplan} gave an
   approach to Conjecture A and Question (\ref{eq1.2}) by counting
   numerical semigroups by genus and multiplicity. He obtained some partial interesting 
results, but his method does not solve the problems.

In addition, Bras-Amor\'os \cite{Amoros3} introduced the notion of 
{\em ordinarization transform} $\T:\cS_g\to\cS_g$ given by 
$\T(S)=(S\cup\{F(S)\})\setminus\{m(S)\}$, with $S\neq 
S_g:=\{0\}\cup\{g+i: i\in \bN\}$ (so-called {\em ordinary semigroup} of 
genus $g$). Then the minimum nonnegative integer $r$ such that 
$\T^r(S)=S_g$ is the {\em ordinarization number} of $S$; it turns 
out that $r\leq g/2$, and so she counted numerical semigroups by 
genus and ordinarization number. Unfortunately this method also does 
not give an answer to either computing $n_g$ or question 
(\ref{eq1.2}).

In this paper we approach (\ref{eq1.2}) by counting numerical semigroups by genus and number 
of even gaps. Our method is motivated by the interplay between double covering of curves and 
Weierstrass semigroups at totally ramified points of such coverings; see for instance Kato 
\cite{Kato}, Garcia \cite{Arnaldo}, Torres \cite{Torres1}, Oliveira and Pimentel 
\cite{Oliveira-Pimentel}, Komeda \cite{Komeda}.

Let $N_\gamma(g)$ denote the number of elements of the family $\cS_\gamma(g)$, 
so-called {\em $\gamma$-hyperelliptic semigroups of genus $g$}; i.e. 
those in $\cS_g$ whose number of even gaps equals $\gamma$. From Corollary \ref{cor2.1}
   \begin{equation}\label{eq1.21}
   n_g=\sum_{\gamma=0}^{\lfloor 2g/3\rfloor}N_\gamma(g)\, ;
   \end{equation}
  in particular, see Remark \ref{rem3.0}, Question (\ref{eq1.2})
   holds true provided that
    \begin{equation}\label{eq1.3}
    N_\gamma(g+1)>N_\gamma(g)\quad\text{for
    $\gamma=\lfloor g/3\rfloor+1, \ldots, \lfloor 2g/3\rfloor $}\, .
    \end{equation}
In Section \ref{s2} we deal with the set of even gaps of a numerical 
semigroup, where the key result is Lemma \ref{lemma2.1} (cf. 
\cite{Torres2}). In particular, (\ref{eq1.21}) is a direct consequence of 
the stratification in (\ref{eq2.2}). For $2g\geq 3\gamma$ (cf. \cite{Strazzanti}) we point 
out a quite useful parametrization, namely $\cS_\gamma(g)\to 
\cS_\gamma$, $S\mapsto S/2$, which was introduced by Rosales et al. 
\cite{Rosales} (see (\ref{eq2.3}), \cite{Aur}, \cite{Gu-Tang}). Thus Remark \ref{rem2.2} shows 
the class of numerical semigroups we deal with in this paper; we do 
observe that these semigroups were already studied for example in 
\cite{GFJ} by using the concept of {\em weight} of semigroups.

We have $N_\gamma(g)\leq N_\gamma(3\gamma)$ and $N_\gamma(g)=N_\gamma(3\gamma)$ if and only 
if $g\geq 3\gamma$; see Corollary \ref{cor3.1}. The key ingredient here is the 
$t$-translation of a numerical semigroup introduced in Definition \ref{def3.1}.

By the above considerations on $N_\gamma(g)$, it is natural to investigate the asymptotic 
behaviour of the sequence $N_\gamma(3\gamma)$ which is studied in Section \ref{s4}; indeed, 
to our surprise, it coincides with the sequence $f_\gamma$, introduced by Bras-Amor\'os in 
\cite[p. 2515]{Amoros3}, which has to do with semigroup-closed sets (see Theorem 
\ref{thm4.1} here).

Finally in Section \ref{s5} we compute certain limits involving 
$f_\gamma$ (see Proposition \ref{prop5.1}) which are of theoretical 
interest as they are related to the stronger possibility: 
$f_\gamma\sim \varphi^{2\gamma}$.

   \section{On the even gaps of a numerical semigroup}\label{s2}

   Throughout, let $S$ be a numerical semigroup of genus
   $g=g(S)$, $G_2=G_2(S)$ the set of its even gaps, and
   $\gamma=\gamma(S)$ the number of elements of $G_2$. As a matter of
   terminology, we say that $S$ is {\em
   $\gamma$-hyperelliptic}. In particular, from (\ref{eq1.1}),
   there are exactly $g-\gamma$ (resp. $\gamma$) even (resp.
   odd) nongaps in $S\cap [1,2g]$. For $\gamma\geq 1$, these
   odd nongaps will be denoted by
      \begin{equation}\label{eq2.1}
   o_\gamma=o_\gamma(S)<\ldots<o_1=o_1(S)\, .
   \end{equation}
   \begin{remark}\label{rem2.0} With notation as above, we notice that 
   $o_i\leq 2g-2i+1$ for $i=1,\ldots,\gamma$.
   \end{remark}
   As usual, for pairwise different natural numbers $a_1,\ldots, a_\alpha$, we set $
\langle a_1,\ldots,a_\alpha\rangle:=
\{a_1x_1+\ldots+a_\alpha x_\alpha: x_1, \ldots, x_\alpha\in 
\bN_0\}$. It is well-known, so far, that this set is a numerical 
semigroup if and only if $\gcd(a_1,\ldots,a_\alpha)=1$.
      \begin{remark}\label{rem2.1} We have $\gamma(S)=0$ if
    and only $S=\langle 2,2g+1\rangle $; in the literature,
    this semigroup is classically called {\em hyperelliptic}.
    In general $g\geq \gamma$, and equality holds if and only if 
$g=\gamma=0$.
    \end{remark}
    From now on, we always assume $\gamma\geq 1$ so that
    $1,2\in G(S)$, the set of gaps of $S$, and $g\geq\gamma+1$.

The following result and their corollaries were already noticed in 
\cite{Torres2}. It is analogous to (\ref{eq1.1}), and for the sake 
of completeness we state proofs.
   \begin{lemma}\label{lemma2.1} The biggest even gap $\ell$
   of a $\gamma$-hyperelliptic semigroup $S$ of genus $g$ satisfies
   $$
  \ell\leq \min(4\gamma-2,4g-4\gamma)\, .
    $$
    \end{lemma}
       \begin{proof} Suppose that $\ell\geq 4\gamma$. Then in the
   interval $[2,4\gamma-2]$ there are at least $\gamma$ even
   nongaps of $S$ says, $h_1<\ldots<h_\gamma$. Thus $S$ would
   have at least $\gamma+1$ even gaps, namely
   $\ell-h_\gamma<\ldots<\ell-h_1<\ell$, a contradiction.

Now if $4\gamma-2\leq 4g-4\gamma$; i.e., $g\geq 2\gamma$, the proof 
follows. Otherwise, consider $I:=2\gamma-g+1$ which is a positive 
integer with $I\leq \gamma$ as $g\geq \gamma+1$. Suppose that 
$\ell>o_I$, being $(o_j)$ the sequence of odd nongaps of $S$ in 
(\ref{eq2.1}). Thus we obtain $\gamma-I+1=g-\gamma$ odd gaps of $S$, 
namely
   $$
   \ell-o_I<\ldots< \ell-o_\gamma\, ;
   $$
hence $\ell=o_I+1\leq 2g-2I+2=4g-4\gamma$ (cf. Remark \ref{rem2.0}), and the result follows.
   \end{proof}
   \begin{corollary}\label{cor2.1} {\rm (cf. \cite{Strazzanti})} Let $S$ be a
   $\gamma$-hyperelliptic semigroup of genus $g.$ Then $2g\geq
   3\gamma.$
   \end{corollary}
   \begin{proof} If $g\geq 2\gamma$, the result is clear. Let
   $g\leq 2\gamma-1$. By Lemma \ref{lemma2.1}, $G_2$ is
   contained in the interval $[2,4g-4\gamma]$ and hence
   $2g-2\gamma\geq \gamma$ and we are done.
   \end{proof}
   \begin{corollary}\label{cor2.2} Let $S$ be a
   $\gamma$-hyperelliptic semigroup$.$ Then its smallest odd nongap 
   $O:=o_\gamma(S)$ satisfies $O\geq \max(|2g-4\gamma|+1,3).$
   \end{corollary}
   \begin{proof} Clearly $O\geq 3$ since $\gamma\geq 1$. Let
   $g\geq 2\gamma$. By Lemma \ref{lemma2.1} in
   $[2,4\gamma]\cap S$ there are exactly $\gamma$ even nongaps, say
    $h_1<\ldots<h_\gamma=4\gamma$. Thus the elements $O$ and
    $O+h_j$ are $\gamma+1$ odd nongaps of $S$.
    Since $S$ has exactly $\gamma$ odd nongaps in $[1,2g-1]$, then
    $O+4\gamma\geq 2g+1$ and the result follows.

    Now let $g\leq 2\gamma-1$. Here, by Lemma \ref{lemma2.1}, in
    $[2,4g-4\gamma]\cap S$ there are exactly $2g-3\gamma$ even
    nongaps. Consider the sequence
    $2o_\gamma<\ldots<o_\gamma+o_{4\gamma-2g}$ of
    $2g-3\gamma+1$ elements. Thus $o_\gamma+o_{4\gamma-2g}\geq
    4g-4\gamma+2$. Since $o_{4\gamma-2g}\leq 6g-8\gamma+1$ by Remark \ref{rem2.0}, we are done.
    \end{proof}
    As a way of illustration, next we describe $1$-hyperelliptic
    and $2$-hyperelliptic semigroups.
       \begin{example}\label{ex2.1} Let $\gamma=1$ and thus
    $g\geq 2$. Then $G_2=\{2\}$ by Lemma
    \ref{lemma2.1} and $o_1\geq \max(2g-3,3)$ by Corollary
    \ref{cor2.2}. Thus we obtain two types of
    $1$-hyperelliptic semigroups of genus $g$, namely
    $\langle 4,6,2g-3\rangle$
    with $g\geq 3$, and $\langle 4,6,2g-1,2g+1\rangle $
    with $g\geq 2$.
       \end{example}
       \begin{example}\label{ex2.2} Let $\gamma=2$ and hence
    $g\geq 3$. Let $g=3$. Then $G_2=\{2,4\}$ by Lemma
    \ref{lemma2.1} and thus $S=\langle 3, 5,7\rangle$. Now let
    $g\geq 4$. By Lemma \ref{lemma2.1} there is missing just
    one even nongap in $S\cap [4,6]$, and by Corollary \ref{cor2.2}
    $o_2\geq
    \max(2g-7,3)$. For $g=4$ we have the following four possibilities 
of $2$-hyperelliptic semigroups:
    $\langle 3,5\rangle$, $\langle 3,7,8\rangle$,
    $\langle 4,5,7\rangle$, $\langle 5,6,7,8,9\rangle$.

    So let $g\geq 5$ and thus $o_2\geq 2g-7$. Here we obtain the 
following seven
    families of $2$-hyperelliptic semigroups of genus $g$:
       \begin{enumerate}
    \item[\rm(1)] $\langle 4,10,2g-7\rangle$ with $g\geq 6$;
    \item[\rm(2)] $\langle 4,10,2g-5,2g+1\rangle$;
    \item[\rm(3)] $\langle 4,10,2g-3,2g-1\rangle$;
    \item[\rm(4)] $\langle 6,8,10,2g-7\rangle$;
    \item[\rm(5)] $\langle 6,8,10,2g-5,2g-3\rangle$;
    \item[\rm(6)] $\langle 6,8,10,2g-5,2g-1\rangle$;
    \item[\rm(7)] $\langle 6,8,10,2g-3,2g-1,2g+1\rangle$.
    \end{enumerate}
      \end{example}
      \begin{remark}\label{rem2.11} The examples above were already
      handled, among others, by Garcia \cite{Arnaldo} and Oliveira-Pimentel
      \cite{Oliveira-Pimentel} who moreover noticed that
   all of them are Weierstrass semigroups; this property is also
  true for $3$-hyperelliptic curves (see Komeda \cite{Komeda}).
  We point out that there are numerical semigroups which are not
  Weierstrass; cf. \cite{Torres1}.
     \end{remark}
The following computations have to do with Corollary \ref{cor2.1}.
      \begin{example}\label{ex2.3} {\rm (cf. \cite{Strazzanti})} We look for
    $\gamma$-hyperelliptic semigroups $S$ of genus $g$ such that
    $g=\lceil 3\gamma/2\rceil$.

    {\bf Case $\gamma$ even.} For example for $\gamma=2$ and
    $g=3$, $S=\langle 3, 5, 7\rangle$, as one can easily see from
    Example \ref{ex2.2}. In general, we show that $S$
    is generated by the set $\Sigma:=\{\gamma+2i-1:
    i=1,\ldots,\gamma+1\}$. Indeed, here $o_\gamma\geq \gamma+1$
   by Corollary \ref{cor2.2}. Since in
   $[\gamma+1,2g-1]$ there are exactly $\gamma$ odd numbers, then 
   the $g-\gamma=\gamma/2$ odd gaps of $S$ are precisely
   the odd numbers in $[1,\gamma-1]$. On the other hand,
   Lemma \ref{lemma2.1} implies $G_2(S)\subseteq [2, 2\gamma]$ so 
that the $g-\gamma$ even numbers in $[2\gamma+2,2g]$ are even 
nongaps. Thus 
$G(S)=\{2i:i=1,\ldots,\gamma\}\cup\{2i-1:i=1,\ldots,\gamma/2\}$,
   or equivalently, $S$ is generated by $\Sigma$ as follows
   from e.g. \cite[Sect. 3(III)]{Selmer}.

{\bf Case $\gamma$ odd.} Here $2g=3\gamma+1$,
   $4g-4\gamma=2\gamma+2$. If $\gamma=1$ and hence $g=2$,
   Example \ref{ex2.1} shows that $S=\langle 3,4,5\rangle $.
   Let $\gamma\geq 3$ and so Lemma \ref{lemma2.1} implies
   $G_2\subseteq [2,2\gamma+2]$. Since in
   $[2\gamma+4,3\gamma+1]$ we have $(\gamma-1)/2=g-\gamma-1$
   even numbers, $S$ has just one even nongap $x$
   missing in
   the interval $[\gamma+3, 2\gamma+2]$. This gives
   $(\gamma+1)/2$ possibilities for the selection of $x\, (*)$ so
   that the even nongaps in $S\cap [2,2g]$ are the elements
   $\{x\}\cup\{2\gamma+2+2i:i=1,\ldots, (\gamma-1)/2\}$.

   Next we look for the odd nongaps of $S$; we have that
   $\gamma\leq o_\gamma\leq \gamma+2$ by Corollary \ref{cor2.2}
   and the definition of $o_\gamma$.

   {\bf 1.} Let $o_\gamma=\gamma+2$. In the interval
   $[\gamma+2,3\gamma]$ there are precisely $\gamma$ odd integers and
   thus the set of odd nongaps of $S$ in $S\cap [1,2g]$ is
   $\{\gamma+2i:i=1,\ldots,\gamma\}$.

   Then for each $\gamma$ odd we obtain $(\gamma+1)/2$
   $\gamma$-hyperelliptic semigroups of genus $g=(3\gamma+1)/2$.

   {\bf 2.} Let $o_\gamma=\gamma$. In this case $x=2\gamma$ in $(*)$
   above. In the interval $[\gamma+2,3\gamma]$ there are $\gamma$ odd
   numbers from which we have to choose $\gamma-1$ of them. If
   $\gamma+2\in S$, $2\gamma+2\in S$, a contradiction. Thus the odd
   nongaps in $S\cap [1,2g]$ are determined, namely those in the set
   $\{\gamma\}\cup\{\gamma+3+2i-1:i=1,\ldots,\gamma-1\}$; i.e. we 
just
   obtain one numerical semigroup in this case.
       \end{example}
   Now we study a natural stratification of the family $\cS_g$
   defined above, by taking into
   consideration even gaps. As a matter of fact, we collect the
   subfamily of $\gamma$-hyperelliptic semigroups of genus $g$:
   $$
\cS_{\gamma}(g):=\{S\in\cS_g: \gamma(S)=\gamma\}\, ,
  \quad\text{and thus}
   $$
   \begin{equation}\label{eq2.2}
   \cS_g=\bigcup_{\gamma=0}^{\lfloor 2g/3\rfloor}S_\gamma(g)
   \end{equation}
   by Corollary \ref{cor2.1}.
   The following definition was introduced by Rosales et al.
   \cite{Rosales} in connection with Diophantine inequalities;
   see also \cite[p. 371]{Torres2}, or the proof of
   \cite[Scholium 3.5]{Torres1}, where this concept is
   related to St\"ohr's examples concerning symmetric semigroups
   which are not Weierstrass semigroups.
      \begin{definition}\label{def2.1} The {\em one half} of a 
numerical semigroup $S$ is $S/2:=\{s\in \bN_0: 2s\in S\}.$
   \end{definition}
We notice that $g(S/2)=\gamma(S)$; in particular, we have a natural 
parametrization of the family 
$\cS_{\gamma}(g)$ onto $\cS_\gamma$, where $2g\geq 3\gamma$, by 
means of the function
    \begin{equation}\label{eq2.3}
  \x=\x_\gamma(g): \cS_\gamma(g)\to \cS_\gamma\, ,\quad S
  \mapsto S/2\, .
    \end{equation}
    This map is certainly surjective: Let $T\in \cS_\gamma$, then 
$\x(S)=T$, where
   $$
S:=2T\cup\{2g-2\gamma+i: i\in \bN\}\in \cS_{\gamma}(g)\, ,
   $$
being $2T:=\{2t:t\in T\}$; see also \cite{DAS}.
   \begin{remark}\label{rem2.2} Indeed, any $S\in \cS_{\gamma}(g)$
   can be uniquely written as being:
   $$
S=2(S/2)\cup\{o_\gamma<\ldots<o_1\}\cup \{2g+i:i\in\bN_0\}\, ,
   $$
   where $o_\gamma,\ldots,o_1$ are certain odd numbers in
   $[O,2g-1]$ (cf. Remark \ref{rem2.0}) with $O=\max\{|2g-4\gamma|+1,3\}$ by Corollary 
\ref{cor2.2}. See also \cite{Aur}, \cite{Gu-Tang}.
   \end{remark}

   \section{On the family $\cS_{\gamma}(g)$}\label{s3}

   In this section we deal with the family $\cS_{\gamma}(g)$ of 
numerical semigroups of genus $g$ whose number of even gaps equals $\gamma$.
   Throughout we assume $2g\geq 3\gamma$ (cf. Corollary 
\ref{cor2.1}).
     \begin{definition}\label{def3.1} Let $t\in\bZ.$ The
     $t$-translation of a numerical semigroup $S$ is the map
     $\Phi_t:S\to \bZ$ defined by
   $$
   s\mapsto \begin{cases} s & \text{if $s\equiv 0\pmod{2}$}\, ,\\
   s-t & \text{otherwise}\, .\\
   \end{cases}
   $$
   \end{definition}
      \begin{lemma}\label{lemma3.1} Let $t=2g-6\gamma,$ $S\in\cS_\gamma(g).$ Then $\Phi_t(S)\in 
   \cS_\gamma(3\gamma).$ 
   \end{lemma} 
   \begin{proof} We first show that $\Phi(S):=\Phi_t(S)$ 
   is indeed a numerical semigroup. By Lemma \ref{lemma2.1} it is enough to notice that 
   $2(o_\gamma(S)-t)\geq 4\gamma+2$ which is clear from Corollary \ref{cor2.2} and the 
   selection of $t$. In particular, $\Phi(S)\subseteq \bN_0$ with 
   $\gamma(\Phi(S))=\gamma$. Next we show that $g(\Phi(S))=3\gamma$.

   Let $x=2g+i\in S$, $i\in\bN_0$. Then $x-t=6\gamma+i\in
   \Phi(S)$ so that $y\in \Phi(S)$ for all $y\geq 6\gamma$. In $[1,6\gamma-1]$ we have
   $(3\gamma-\gamma)=2\gamma$ odd gaps of $\Phi(S)$; hence $g(\Phi(S))=3\gamma$.
        \end{proof}
      Thus Definition \ref{def3.1} with $t=2g-6\gamma$ induces a map
      $$
      \tilde\Phi_t: \cS_{\gamma}(g)\to \cS_\gamma(3\gamma)\, ,
      \quad S\mapsto\Phi_t(S)\, .
      $$
          \begin{theorem}\label{thm3.1} The map $\tilde\Phi_t$ above is injective$,$ and it is 
          bijective if and only if $g\geq 3\gamma.$ 
        \end{theorem}
            \begin{proof} The map $\tilde\Phi_t$ is
        injective by its definition. 

             Let $g<3\gamma$ so $t=2g-6\gamma\leq -2$. Then the map $\tilde\Phi_t$ is not 
     surjective. Indeed, let $S:=\langle 4,2\gamma+1\rangle$ which belongs to 
     $\cS_\gamma(3\gamma)$. Suppose there exists
     $T\in\cS_\gamma(g)$ such that $\tilde\Phi_t(T)=S$.
     Then $o_\gamma(T)=2\gamma+1+t$ so that $h:=2o_\gamma(T)=
     4\gamma+2+2t\in T$ with $h\in \Sigma:=\{\ell\in \bN: \ell
     \equiv 2\pmod{4}\, ,2\leq \ell\leq 4\gamma-2\}$.
     It turns out that $\Sigma \subseteq G(T)$, a contradiction.

Conversely, for $T\in\cS_\gamma(3\gamma)$ let us consider the $(-t)$-translation 
$\Phi_{(-t)}:T\to
        \bZ$. Here we have $2( o_\gamma(T)+t)\geq 2(2\gamma+1+2g-6\gamma)\geq 4\gamma+2$ by 
        Corollary \ref{cor2.2} and $g\geq 3\gamma$. Thus we have a map $\tilde\Phi_{(-t)}: 
        \cS_\gamma(3\gamma)\to \cS_\gamma(g)$ induced by $\Phi_{(-t)}$ which is clearly the 
        inverse of $\tilde\Phi_t$.
           \end{proof}
Recall that $N_\gamma(g)=\#\cS_{\gamma}(g)$.
     \begin{corollary}\label{cor3.1}\quad $N_\gamma(g)\leq
     N_\gamma(3\gamma);$ equality holds if and only if 
     $g\geq 3\gamma.$ 
     \end{corollary} 
     \begin{remark}\label{rem3.0} Here we explain how a positive answer to 
     Question \ref{eq1.2} would follow from inequalities (\ref{eq1.3}). 
     From Corollary \ref{cor2.1}, $N_\gamma(g)=0$ for $\gamma>\lfloor 2g/3 \rfloor$ and from 
     Corollary \ref{cor3.1}, $N_\gamma(g)=N_\gamma(3\gamma)$ for $\gamma\leq 
     \lfloor g/3\rfloor$. Indeed, this is the 
     best information that we can obtain by using the map $\Phi_t$ above. In particular,
     $$
n_{g+1}=\sum_{\gamma=0}^{\lfloor (g+1)/3\rfloor}N_\gamma(3\gamma)+
\sum_{\gamma=\lfloor (g+1)/3\rfloor+1}^{\lfloor 2(g+1)/3\rfloor} N_\gamma(g+1)\quad\text{and}
    $$
    $$
 n_g=\sum_{\gamma=0}^{\lfloor g/3\rfloor}N_\gamma(3\gamma)+
\sum_{\gamma=\lfloor g/3\rfloor +1}^{\lfloor 2g/3\rfloor} N_\gamma(g)
   $$
   so that
   $$
n_{g+1}-n_g\geq \sum_{\gamma=\lfloor g/3\rfloor+1}^{\lfloor 2g/3\rfloor}(N_\gamma(g+1)-N_\gamma(g))
   $$
    and (\ref{eq1.3}) implies (\ref{eq1.2}).
  \end{remark}     
     Next we display two tables for some values of 
$N_\gamma(g)$ which show that (\ref{eq1.3}) might be true; 
we obtain such computations by using the GAP package \cite{GAP}. 
    \newpage

   \begin{table}[h]
\centering
\begin{tabular}{|c|c c c c c c c c c c|}
  \hline
\diagbox[height=0.6cm]{$g$}{$\gamma$} & $0$ & $1$ & $2$ & $3$ & $4$ & $5$ & $6$ & $7$ & 8 & 9 
\\
\hline
0 & 1 &  &  &  &  &  &  &  &  &  \\ 
\hline
1 & 1 &  &  &  &  &  &  &  &  &  \\
\hline
2 & 1 & 1 &  &  &  &  &  &  &  &  \\
\hline
3 & 1 & 2 & 1 &  &  &  &  &  &  &  \\
\hline
4 & 1 & 2 & 4 &  &  &  &  &  &  &  \\
\hline
5 & 1 & 2 & 6 & 3 &  &  &  &  &  &  \\
\hline
6 & 1 & 2 & 7 & 12 & 1 &  &  &  &  &  \\
\hline
7 & 1 & 2 & 7 & 19 & 10 &  &  &  &  &  \\
\hline
8 & 1 & 2 & 7 & 21 & 32 & 4 &  &  &  &  \\
\hline
9 & 1 & 2 & 7 & 23 & 51 & 33 & 1 &  &  &  \\
\hline
10 & 1 & 2 & 7 & 23 & 62 & 91 & 18 &  &  &  \\
\hline
11 & 1 & 2 & 7 & 23 & 65 & 142 & 98 & 5 &  &  \\
\hline
12 & 1 & 2 & 7 & 23 & 68 & 174 & 257 & 59 & 1 &  \\
\hline
13 & 1 & 2 & 7 & 23 & 68 & 192 & 412 & 271 & 25 &  \\
\hline
14 & 1 & 2 & 7 & 23 & 68 & 197 & 514 & 678 & 197 & 6  \\
\hline
15 & 1 & 2 & 7 & 23 & 68 & 200 & 570 & 1100 & 793 & 92 \\
\hline
16 & 1 & 2 & 7 & 23 & 68 & 200 & 602 & 1409 & 1855 & 606  \\
\hline
17 & 1 & 2 & 7 & 23 & 68 & 200 & 609 & 1595 & 2999 & 2191  \\
\hline
18 & 1 & 2 & 7 & 23 & 68 & 200 & 615 & 1693 & 3890 & 4993  \\
\hline
19 & 1 & 2 & 7 & 23 & 68 & 200 & 615 & 1744 & 4472 & 8126  \\
\hline
20 & 1 & 2 & 7 & 23 & 68 & 200 & 615 & 1756 & 4797 & 10723  \\
\hline
21 & 1 & 2 & 7 & 23 & 68 & 200 & 615 & 1764 & 4959 & 12528  \\
\hline
22 & 1 & 2 & 7 & 23 & 68 & 200 & 615 & 1764 & 5034 & 13616  \\
\hline
23 & 1 & 2 & 7 & 23 & 68 & 200 & 615 & 1764 & 5053 & 14191  \\
\hline
24 & 1 & 2 & 7 & 23 & 68 & 200 & 615 & 1764 & 5060 & 14469  \\
\hline
25 & 1 & 2 & 7 & 23 & 68 & 200 & 615 & 1764 & 5060 & 14589  \\
\hline
26 & 1 & 2 & 7 & 23 & 68 & 200 & 615 & 1764 & 5060 & 14611  \\
\hline
27 & 1 & 2 & 7 & 23 & 68 & 200 & 615 & 1764 & 5060 & 14626  \\
\hline
   \end{tabular}
\vspace{0.3cm}
\caption{A few values for $N_{\gamma}(g)$}
   \end{table}
   \newpage

   \begin{table}[h]
\centering
\begin{tabular}{|c|c c c c c c c c c| c|}
  \hline
\diagbox[height=0.6cm]{$g$}{$\gamma$} & 10 & 11 & 12 & 13 & 14 & 15 & 16 & 17 & 18 & $n_g$ 
\\
\hline
0 &  &  &  &  &  &  &  &  &  & 1 \\ 
\hline
1 &  &  &  &  &  &  &  &  &  & 1 \\
\hline
2 &  &  &  &  &  &  &  &  &  & 2 \\
\hline
3 &  &  &  &  &  &  &  &  &  & 4 \\
\hline
4 &  &  &  &  &  &  &  &  &  & 7 \\
\hline
5 &  &  &  &  &  &  &  &  &  & 12 \\
\hline
6 &  &  &  &  &  &  &  &  &  & 23 \\
\hline
7 &  &  &  &  &  &  &  &  &  & 39 \\
\hline
8 &  &  &  &  &  &  &  &  &  & 67 \\
\hline
9 &  &  &  &  &  &  &  &  &  & 118 \\
\hline
10 &  &  &  &  &  &  &  &  &  & 204 \\
\hline
11 &  &  &  &  &  &  &  &  &  & 343 \\
\hline
12 &  &  &  &  &  &  &  &  &  & 592 \\
\hline
13 &  &  &  &  &  &  &  &  &  & 1001 \\
\hline
14 &  &  &  &  &  &  &  &  &  & 1693 \\
\hline
15 & 1 &  &  &  &  &  &  &  &  & 2857 \\
\hline
16 & 33 &  &  &  &  &  &  &  &  & 4806 \\
\hline
17 & 343 & 7 &  &  &  &  &  &  &  & 8045 \\
\hline
18 & 1836 & 138 & 1 &  &  &  &  &  &  & 13467 \\
\hline
19 & 6033 & 1130 & 43 &  &  &  &  &  &  & 22464 \\
\hline
20 & 13317 & 5335 & 544 & 8 &  &  &  &  &  & 37396 \\
\hline
21 & 21764 & 16447 & 3624 & 191 & 1 &  &  &  &  & 62194 \\
\hline
22 & 29209 & 35392 & 15365 & 1897 & 53 &  &  &  &  & 103246 \\
\hline
23 & 34628 & 57925 & 44575 & 11098 & 804 & 9 &  &  &  & 170963 \\
\hline
24 & 38096 & 78602 & 93919 & 43262 & 6485 & 254 & 1 &  &  & 282828 \\
\hline
25 & 40098 & 94469 & 154077 & 119669 & 33525 & 3013 & 64 &  &  & 467224 \\
\hline
26 & 41086 & 105074 & 211576 & 247756 & 120881 & 20945 & 1153 & 10 &  & 770832 \\
\hline
27 & 41541 & 111426 & 257734 & 407238 & 320649 & 98104 & 10873 & 335 & 1 & 1270267 \\
\hline
\end{tabular}
\vspace{0.3cm}
\caption{A few values for $N_{\gamma}(g)$ (cont.)}
  \end{table}

We end up this section by pointing out a result concerning specific 
properties of semigroups $S$ in the fiber $\x^{-1}(T)$ in 
(\ref{eq2.3}), where $T\in \cS_\gamma$. For example, for 
$g,\gamma\in \bN_0$ with $g\geq 3\gamma$, let us consider St\"ohr's 
examples in \cite[p. 48]{Torres1}:
     $$
     S:=2T\cup\{2g-1-2t: t\in\bZ\setminus T\}
     $$
     which are $\gamma$-hyperelliptic symmetric semigroups of genus
     $g$.
     Thus we have:
    \begin{scholium}\label{scholium3.1} Let $g$ and $\gamma$ be
    integers such that $g\geq 3\gamma.$ Then there exists, at least, 
$n_\gamma$ $\gamma$-hyperelliptic symmetric semigroups of genus $g.$
     \end{scholium}

    \section{On the sequence $f_\gamma$}\label{s4}

This section is closely related to Bras-Amor\'os approach 
\cite{Amoros3}; see Theorem \ref{thm4.1}.
   \begin{definition}\label{def4.1} Let $S$ be a numerical
   semigroup.
   \begin{enumerate}
\item[\rm(1)] A set $B\subseteq \bN_0$ is called {\em $S$-closed} if 
for $b\in B$, $s\in S$ we have either $b+s\in B$, or $b+s>\max(B)$.
   \item[\rm(2)] We let $C(S,i)$ denote the collection of
   $S$-closed sets $B$ such that $0\in B$ and $\#B=i$.
   \end{enumerate}
   \end{definition}

   \begin{lemma}\label{lemma4.1} Let $S\in \cS_\gamma,$
   $B\in C(S,\gamma+1).$ Then $\max(B)\leq 2\gamma.$
   \end{lemma}

   \begin{proof} Suppose $\max(B)>2\gamma$. Then
   $F:=\#[0,2\gamma]\cap S\leq \gamma$; since $g(S)=\gamma$,
   $F=\gamma+1$ which gives rise to a contradiction.
    \end{proof}

   \begin{definition}\label{def4.2} {\rm (\cite{Amoros3})}
   For $\gamma\in \bN_0,$ $f_\gamma:= 
\sum_{S\in\cS_\gamma}\#C(S,\gamma+1)$.
   \end{definition}

The main result of this section is the following. Notation as in 
Section \ref{s3}.

   \begin{theorem}\label{thm4.1} For $\gamma\in\bN_0,$
   $f_\gamma=N_\gamma(3\gamma)=\#\cS_\gamma(3\gamma).$
   \end{theorem}

   \begin{proof} Let $\x=\x_\gamma(3\gamma): \cS_\gamma(3\gamma)\to 
\cS_\gamma$,
   $S\mapsto S/2$ (see Definition \ref{def2.1}). The result follows 
from the
   following computations.

   {\bf Claim.} There is a bijective map $\F$ between the sets 
$C(T,\gamma+1)$ and $\x^{-1}(T)$ with $T\in\cS_\gamma$.

In fact, for $B\in C(T,\gamma+1)$ we let 
$\F(B)=2T\cup\{2b-2\max(B)+6\gamma+1: b\in B\}$; this map is well 
defined by Lemma \ref{lemma4.1} as a similar proof to the one of Lemma \ref{lemma3.1} shows.

Now let $S\in \x^{-1}(T)$ so that 
$S=2T\cup\{o_\gamma(S)<\ldots<o_1(S)<o_0:=6\gamma+1\}\cup 
\{6\gamma+i:i\in\bN_0\}$ with $o_i$ odd integers. Set $o_i(S)=o_i$ 
and define $b_i:=(o_i-o_\gamma)/2$, $i=0,\ldots,\gamma$. By 
definition it is clear that $B:=\{b_0,\ldots,b_\gamma\}\in 
C(T,\gamma+1)$ and the inverse map of $\F$ is given by $S\mapsto B$.
    \end{proof}
Next we investigate bounds on the sequence $f_\gamma$ by taking 
advantage of Theorem \ref{thm4.1} above; thus we shall be dealing 
with sets of the form:
    \begin{equation}\label{eq4.1}
    S=2T\cup\cO\cup \{6\gamma+j:j\in\bN_0\}\, ,
    \end{equation}
where $T\in \cS_\gamma$, and $\cO=\{o_\gamma<\ldots<o_1\}$ is 
certain set of $\gamma$ odd integers in $[2\gamma+1,6\gamma-1]$.
      \begin{remark}\label{rem4.1} The set $S$ in (\ref{eq4.1})
      belongs to $\cS_\gamma(3\gamma)$ if and only if for
   $t\in T$, $o_j\in\cO$ we have $2t+o_j\in\cO$ or $2t+o_j>6\gamma$.
    \end{remark}
Throughout, we let
   $$
   o_\gamma=2\gamma+2i+1\quad\text{for some
    $i\in\{0,\ldots,\gamma\}$}\, .
    $$
    In addition we set:
    \begin{equation}\label{eq4.2}
    \x^{-1}(T^i):=\{S\in \cS_\gamma(3\gamma): S/2=T,\, o_\gamma(S)=
    2\gamma+2i+1\}\, ,
    \end{equation}
where $\x=\x_\gamma(g)$ is the map in (\ref{eq2.3}) with 
$g=3\gamma$. We notice that 
$\x^{-1}(T)=\cup_{i=0}^\gamma\x^{-1}(T^i)$, and 
$\cS_\gamma(3\gamma)=\cup_{T\in\cS_\gamma}\x^{-1}(T)$.
   \begin{lemma}\label{lemma4.2} Let $T\in\cS_\gamma.$
   With the above notation$,$
      $$
    1\leq  \#\x^{-1}(T^i)\leq \binom{\gamma}{i}\, .
      $$
    \end{lemma}
    \begin{proof} Let $t_0=0<t_1<\ldots$ be the
    enumeration of $T$ in increasing order. By (\ref{eq1.1})
    $t_{\gamma+j}=2\gamma+j$ for any $j\in\bN_0$. Set
    $$
    \cO(1)=\{o_\gamma+2t:t\in T,\, t\leq 2\gamma-i-1\}\, .
    $$
    In (\ref{eq4.1}) let us write $\cO$ as the
    disjoint union of $\cO(1)$ and certain set $\cO(2)$.
   If the elements of $\cO(2)$ are the largest odd integers in
      $$
      [2\gamma+1,6\gamma-1]\setminus\cO(1)\, ,
      $$
      then $S$ in (\ref{eq4.1}) belongs to $S_\gamma(3\gamma)$
      and $\#\x^{-1}(T^i)\geq 1$.
      Since $t_{\gamma-i-1}\leq 2\gamma-i-1$, so $\#\cO(2)\leq i$ and 
the upper bound follows.
    \end{proof}
     \begin{remark}\label{rem4.2} The set $\cO$ for both the
     extreme cases
     $i=0,\gamma$ in Lemma \ref{lemma4.2}
  is easy to describe. In fact here we have
  $\cO= \{2\gamma+1+2t_j:
  j=0,\ldots,\gamma\}$ (resp. $\cO=
  \{4\gamma+2j-1:j=1,\ldots \gamma\}$) for $i=0$ (resp.
  $i=\gamma$). Thus
    $$
    \#\x^{-1}(T^0) = \#\x^{-1}(T^\g) = 1\, .
    $$
   \end{remark}
    From now on, unless otherwise stated, we consider $1\leq
    i\leq\gamma-1$.
   \begin{corollary}\label{cor4.1} Let $\gamma\in\bN_0.$ Then
   $$
   n_{\g} \cdot (\g+1) \leq f_{\g} \leq n_{\g} \cdot 2^{\g}\, .
   $$
   \end{corollary}
   \begin{proof} It follows from
   Theorem \ref{thm4.1}, Lemma \ref{lemma4.2} and the well-known
   fact $\sum_{i=0}^{\gamma}\binom{\gamma}{i}=2^\gamma$.
  \end{proof}
   \begin{remark}\label{rem4.3} From Corollaries \ref{cor3.1} and
   \ref{cor4.1},
   $$
N_\gamma(g)\leq N_\gamma(3\gamma)=f_\gamma\leq n_\gamma\cdot 
2^\gamma\, .
   $$
   \end{remark}
   \begin{corollary}\label{cor4.2} Let $\varphi=(\sqrt{5}+1)/2$
   be the golden ratio$.$
      \begin{enumerate}
   \item[\rm(1)] For $\e > 0,$
   $\lim_{\g \to \infty} \frac{f_{\g}}{(2\f+\e)^{\g}} = 0;$
   \item[\rm(2)] We have $\lim_{\gamma\to\infty}
   \frac{f_\gamma}{\varphi^\gamma}=\infty.$
   \end{enumerate}
   \end{corollary}
   \begin{proof} For $\gamma\in\bN_0$, let 
$A_\gamma:=\frac{n_{\g}}{\f^{\g}}$, $B_\gamma:=
\L(\frac{2\f}{2\f+\e}\R)^{\g}$. We notice that $\lim_{\gamma\to 
\infty}A_\gamma$ is a real number by \cite[Thm. 1]{Zhai}.

(1) By Corollary \ref{cor4.1} $
f_{\g}/(2\f+\e)^{\g} \leq A_\gamma\cdot B_\gamma$. Since 
$\lim_{\gamma\to \infty}B_\gamma=0$, the proof follows.

(2) From Corollary \ref{cor4.1} $f_\gamma/\varphi^\gamma\geq 
A_\gamma(\gamma+1)$, and we are done.
    \end{proof} In the remainder part of this section we shall be 
dealing with Remark \ref{rem4.1} toward an improvement of Corollary 
\ref{cor4.1} (see Corollary \ref{cor4.4} below). We start by 
splitting off the set of odd integers in the interval $[2\g+1,6\g-1]$ 
into $\cO$ and $\cL=\{\omega_1<\ldots<\omega_\gamma\}$. Recall that 
$o_\g = 2\g + 2i + 1$ for some $i\in
  \{1,\ldots,\g-1\}$. Then we have a disjoint union
  $\cL=\cL^i(1)\cup\cL^i(2)$, where
  $$
  \cL^i(1):= \{2\g+1< \ldots< 2\g+2i-1\}\, ,\quad\text{and}
  $$
  $$
\cL^i(2):=\{\o_{i+1}<\ldots<\o_\g\}\subseteq \{o_\gamma+2q: q\in
G(T)\, ,q\leq 2\gamma-i-1\}\, .
  $$
    \begin{lemma}\label{lemma4.3} Let $T\in \cS_\gamma,$ $q,
    \bar q\in G(T)$ such that $\bar q-q=t\in T.$ If $S$
    in (\ref{eq4.1}) is a numerical semigroup$,$ then
    $o_\gamma+2q\in \cL^i(2),$ whenever $o_\gamma+2\bar q\in \cL^i(2).$
    \end{lemma}
    \begin{proof} We have $o_\gamma+2\bar q=o_\gamma+2q+2t$
    so that $o_\gamma+2q\in \cL^i(2)$.
     \end{proof}
Let us work out a numerical example.
  \begin{example}\label{ex4.1} Notation as in (\ref{eq4.2}). For the numerical semigroup
  $T=\bN_0 \setminus \{1,2,3,6\}$ of genus $\gamma=4$, $\#\x^{-1}(T^0)=\#\x^{-1}(T^4)=1$; 
  we shall compute $\#\x^{-1}(T^i)$ for $i=1,2,3$ so that that $\#\x^{-1}(T)=
  \sum_{i=0}^4\#\x^{-1}(T^i)=
  10$. This gives a method to improve Corollary \ref{cor4.1}; as a matter of fact 
  $$
  5n_4+5=5(n_4-1)+10\leq f_4\leq 16n_4-6=16(n_4-1)+10\, .
   $$
(1) If $i=1$, $o_4=11$,
   $$
   \cL^1(2)=\{\omega_2<\omega_3<\omega_4\}
\subseteq\{11+2q:q\in G(T)\, ,q\leq 6\}=\{13,15, 17, 23\}\, .
   $$
Since $23=11+2\times 6$ by Lemma \ref{lemma4.3} $\cL^1(2)$ can be 
either $\{13,15,17\}$ or $\{13,15, 23\}$; it is a matter of fact 
that these computations define semigroups $S$ in (\ref{eq4.1}) so 
that $\#\x^{-1}(T^1)=2$.

(2) If $i=2$, $o_4=13$,
    $$
  \cL^1(2)=\{\omega_3<\omega_4\}\subseteq\{13+2q:q\in G(T)\, , q\leq 
5\}=\{15,17,19\}\, ,
    $$
    so that $\#\x^{-1}(T^2)=\binom{3}{2}=3$.

(3) If $i=3$, $o_4=15$,
   $$
   \cL^1(2)=\{\omega_4\}\subseteq\{15+2q:q\in G(T)\, ,
q\leq 4\}=\{17,19,21\}\, ,
   $$
   so that $\#\x^{-1}(T^3)=\binom{3}{1}=3$.
   \end{example}
Next we generalize this example. Let $k\in\{0,\ldots,\gamma-1\}$ and 
consider the set
     $$
T=T_k = \bN_0 \setminus \{1,\ldots,\g-1,\g+k\}
     $$
     which is a numerical semigroup of genus $\gamma$ by the
     selection of $k$. We shall compute $\#\x^{-1}(T^i)$ for
     $T=T_k$, $0\leq i\leq \gamma$. With notation as above
        $$
\cL^i(2)\subseteq \{o_\gamma+2q: q\in G(T)\, ,q\leq 2\gamma-i-1\}=
    $$
    $$
    \begin{cases}
\{o_\gamma+2q: 1\leq q\leq \gamma-1\}
\cup\{o_\gamma+2(\gamma+k)\}\, , & \text{if $i+k\leq \gamma-1\, 
,$}\\
\{o_\gamma+2q: 1\leq q\leq \gamma-1\}\, ,& \text{if $i+k>\gamma-1\, 
.$}\\
    \end{cases}
  $$
    \begin{lemma}\label{lemma4.4} Notation as above$.$ Let 
$i\in\{0,\ldots,\gamma\},$ $k\in\{0,\ldots,\gamma-1\}.$
   \begin{enumerate}
\item[\rm(1)] For $T=T_0,$ $\#\x^{-1}(T^i)=\binom{\gamma}{i};$

\item[\rm(2)] Let $k\geq 1,$ $T=T_k.$ If $i+k\leq \gamma-1,$ then
   $$
   \#\x^{-1}(T^i)=\binom{\gamma-k-1}{i}+\binom{\gamma-1}{i-1}\, ;
   $$

   \item[\rm(3)] Let $k\geq 1,$ $T=T_k.$ If $i+k\geq \gamma,$ then
   $$
   \#\x^{-1}(T^i)=\binom{\gamma-1}{i-1}\, .
   $$
   \end{enumerate}
   \end{lemma}
   \begin{proof} (1) For $k=0$ we have
   $\cL^i(2)\subseteq\{o_\gamma+2q: 1\leq q\leq \gamma\}$.
   Thus the number of sets of type as in (\ref{eq4.1}) equals
   $\binom{\gamma}{\gamma-i}=\binom{\gamma}{i}$; all such sets
   belong to $\cS_\gamma(3\gamma)$ by Remark \ref{rem4.1}; the
   result follows.

(2) Here $\cL^i(2)\subseteq\{o_\gamma+2q: 1\leq q\leq \gamma-1\}\cup 
\{o_\gamma+2(\gamma+k)\}$. If $o_\gamma+2(\gamma+k)\in \cL^i(2)$, 
then $o_\gamma+2q\in\cL^i(2)$ by Lemma \ref{lemma4.3}. Thus we 
obtain $\binom{\gamma-k-1}{\gamma-i-k-1}=\binom{\gamma-k-1}{i}$ sets 
of type (\ref{eq4.1}) which belong to $\cS_\gamma(3\gamma)$ by 
Remark \ref{rem4.1}. On the other hand, if $o_\gamma\not\in\cL^i(2)$ 
arguing as above we obtain further 
$\binom{\gamma-1}{\gamma-i}=\binom{\gamma-1}{i-1}$ numerical 
semigroups in $\cS_\gamma(3\gamma)$.

(3) In this case $\cL^i(2)\subseteq\{o_\gamma+2q: 1\leq \gamma-1\}$ 
and arguing as in (1) $\#\x^{-1}(T^i)=
\binom{\gamma-1}{i-1}$.
   \end{proof}
By summing up the computations in Lemma \ref{lemma4.4}, we obtain:
   \begin{corollary}\label{cor4.3} Notation as above$.$
   For $k=0,\ldots,\gamma-1,$
     $$
   \#\x^{-1}(T_k)=2^{\gamma-1-k}(2^k+1)\, .
     $$
     \end{corollary}
   \begin{remark}\label{rem4.31} The weight of the semigroup
   $T=T_k$ is $w_k=k$. We have $w_k\leq \gamma/2$, or $\gamma/2
   <w_k\leq \gamma-1$ and $2\gamma>\gamma+k$; thus $T$ is
   Weierstrass \cite{EH}, \cite{Komeda0}.
   Then the unique element in $\x^{-1}(T^\gamma)$ is
   also Weierstrass by \cite[Prop. 2.4]{Komeda}.
   \end{remark}
Set
    $$
M_\gamma:=\sum_{k=0}^{\gamma-1}\#\x^{-1}(T_k)=
2^{\gamma-1}(\gamma+2)-1\, .
   $$
Then, after some computations, Corollary \ref{cor4.1} can be 
improved as follows.
   \begin{corollary}\label{cor4.4} With notation as above$,$
   $$
c_1(\gamma)  \leq f_\gamma\leq c_2(\gamma)\, ,
    $$
    where 
$c_1(\gamma):=n_\gamma(\gamma+1)+M_\gamma-\gamma(\gamma+1),$ and 
$c_2(\gamma):=n_\gamma\cdot 2^\gamma-(\gamma\cdot 2^\gamma-M_\gamma).$
    \end{corollary}
    \begin{remark}\label{rem4.4} From Corollary \ref{cor4.4},
    \cite{Amoros2} we obtain the following computations.
    \begin{table}[h]
    \centering
    \begin{tabular}{|c|c|c|c|}
     \hline
     $\g$ & $c_1(\g)$ & $f_{\g} = N_\g(3\g)$ & $c_2(\g)$
     \\ 
     \hline                               
     0 & 1 & 1 & 1 \\
     \hline
     1 & 2 & 2 & 2 \\
     \hline
     2 & 7 & 7 & 7 \\
     \hline
     3 & 23 & 23 & 27 \\
     \hline
     4 & 62 & 68 & 95 \\
     \hline
     5 & 153 & 200 & 266 \\
     \hline
     6 & 374 & 615 & 1343 \\
     \hline
     7 & 831 & 1764 & 4671 \\
     \hline
     8 & 1810 & 5060 & 16383 \\
     \hline
     9 & 3905 & 14626 & 52993 \\
     \hline
     10 & 8277 & 41785 & 192513 \\
     \hline
     11 & 17295 & 117573 & 666625 \\
     \hline
     12 & 36211 & 332475 & 2347009 \\
     \hline
     13 & 75271 & 933891 & 8032257 \\
     \hline
     14 & 156256 & 2609832 & 27377665 \\
     \hline
     \end{tabular}
     \vspace{0.3cm}
     \caption{Bounds for $f_{\g}$}
     \end{table}
        \end{remark}
In addition, we improve Corollary \ref{cor4.2}(2) above as follows.
    \begin{corollary}\label{cor4.5}\quad
    ${\rm lim}_{\gamma\to\infty}\frac{f_\gamma}{2^\gamma}=
    \infty.$
    \end{corollary}

     \section{Further results on the sequence $f_\gamma$}\label{s5}

From Corollaries \ref{cor4.2}, \ref{cor4.3}, and the column 
regarding $f_\gamma/f_{\gamma-1}$ in Table 4 below, it seems that 
the following property holds true:
    $$
    \text{{\bf (D:)}\qquad $f_\gamma \sim \varphi^{2\gamma}$\, ,}
    $$
    where as usual $\varphi$ is the golden ratio.
       \begin{table}[h]
    \centering
\begin{tabular}{|c|c|c|c|c|c|c|}
  \hline
  $\g$ & $f_{\g}$ & $n_{2\g}$ & $f_{\g}/f_{\g-1}$ & $f_{\g}/n_{2\g}$
   & $f_{\g+1}/\sum_{i=0}^{\g}f_i$ \\ 
  \hline                               
  0 & 1 & 1 & & 1.00 & 2.00 \\
  \hline
  1 & 2 & 2 & 2.00 & 1.00 & 2.33 \\
  \hline
  2 & 7 & 7 & 3.50 & 1.00 & 2.30 \\
  \hline
  3 & 23 & 23 & 3.29 & 1.00 & 2.06\\
  \hline
  4 & 68 & 67 & 2.96 & 1.01 & 1.98 \\
  \hline
  5 & 200 & 204 & 2.94 & 0.98 & 2.04 \\
  \hline
  6 & 615 & 592 & 3.08 & 1.04 & 1.93\\
  \hline
  7 & 1764 & 1693 & 2.87 & 1.04 & 1.89 \\
  \hline
  8 & 5060 & 4806 & 2.87 & 1.05 & 1.89 \\
  \hline
  9 & 14626 & 13467 & 2.89 & 1.09 & 1.87 \\
  \hline
  10 & 41785 & 37396 & 2.86 & 1.12 & 1.83 \\
  \hline
  11 & 117573 & 103246 & 2.81 & 1.14 & 1.83 \\
  \hline
  12 & 332475 & 282828 & 2.83 & 1.18 & 1.82 \\
  \hline
  13 & 933891 & 770832 & 2.81 & 1.21 & 1.80 \\
  \hline
  14 & 2609832 & 2091030 & 2.79 & 1.25 & \\
  \hline
  \end{tabular}
  \vspace{0.3cm}
  \caption{{}}
  \label{table3}
  \end{table}
We end up by computing some interesting limits
    involving the sequence $f_\gamma$ and which are very
    much related to statement {\bf (D)} above. Recall that
    ${\rm lim}_{g\to\infty}\frac{n_{g+1}}{n_g}=\varphi$ \cite{Zhai}.
       \begin{proposition}\label{prop5.1}
   \begin{enumerate}
   \item[\rm(1)] {\bf (D)} is equivalent to $f_\gamma
   \sim n_{2\gamma};$

   \item[\rm(2)] {\bf (D)} implies ${\rm
   lim}_{\gamma\to\infty}\frac{f_{\gamma+1}}{f_\gamma}=\varphi^2;$

   \item[\rm(3)] If ${\rm lim}_{\gamma\to\infty}
   \frac{f_\gamma+1}{f_\gamma}=\varphi^2,$ then
   ${\rm lim}_{\gamma\to\infty}
   \frac{f_{\gamma+1}}{\sum_{i=0}^{\gamma}f_i}=\varphi.$

     \end{enumerate}
     \end{proposition}
  \begin{proof} (1) By \cite[Thm. 1]{Zhai} $\varphi^{2\gamma}
  \sim n_{2\gamma}$, so the result follows.

    (2) Write $\frac{f_{\g+1}}{f_{\g}} =
    \frac{f_{\g+1}}{n_{2\g+2}} \cdot \frac{n_{2\g+2}}{n_{2\g+1}}
     \cdot \frac{n_{2\g+1}}{n_{2\g}} \cdot
     \frac{n_{2\gamma}}{f_{\g}}$. By (1),
     $\lim_{\g \to \infty} \frac{f_{\g}}{n_{2\g}} = K>0$. Then
    $$
    \lim_{\g \to \infty}
    \frac{f_{\g+1}}{f_{\g}}= K \cdot \f \cdot \f \cdot
    \frac{1}{K} = \f^2\, .
    $$
    (3) Let $0<\e<1/3$.\newline
{\bf Claim.}
   $$
   M:=\frac{1-\e\f^2}{\f^2-(1-\e\f^2)} \leq
   \lim_{\g \to \infty}\frac{f_0 + \ldots + f_{\g}}{f_{\g+1}}
   \leq F:=\frac{1+\e\f^2}{\f^2-(1+\e\f^2)}\, .
   $$
Then (3) follows after letting $\epsilon\to 0$ and from the 
well-known fact that $\varphi^2=\varphi+1$.

{\em Proof of the Claim.} By hypothesis, $\lim_{\g \to 
\infty}f_{\g}/f_{\g+1}=1/\f^2$. Set
    $$
    \g_0(\e) := \min\{i \in
    \N: \frac{1}{\f^2}-\e
    <\frac{f_j}{f_{j+1}}
    < \frac{1}{\f^2}+\e, \, \forall j \geq i\}\, .
    $$
For $\g > \gamma_0=\g_0(\e) = \g_0$ write
    $$
\frac{f_0+\ldots+f_{\g_0-1} + f_{\g_0}+\ldots + f_{\g}}{f_{\g+1}}=
\frac{f_0+\ldots+f_{\g_0-1}}{f_{\g+1}}+\sum_{j=\gamma_0}^\gamma 
A_j\, ,
    $$
    where $A_j=f_j/f_{\gamma+1}$. In particular,
    $$
A_j=\frac{f_j}{f_{j+1}}\cdot\ldots\cdot
\frac{f_\gamma}{f_{\gamma}}<t^{\gamma-j+1}\, ,
    $$
being $t=\frac{1}{\varphi^2}+\epsilon<1$ so that
   $$
   {\rm lim}_{\gamma\to \infty}\sum_{j=\gamma_0}^\gamma A_j=F
   $$
   and we obtain the upper bound. We can prove the
   lower bound in a similar way.
   \end{proof}

     {\bf Acknowledgment.} The authors were partially supported respectively by 
CAPES/CNPq-Brazil (grant 140292/2015-2), and CNPq-Brazil (grant 308326/2014-8). They would 
like to thank Pedro A. Garc\'{\i }a-S\'anchez for the computations involving the sequence 
$N_\gamma(g)$, and Maria Bras-Amor\'os and Klara Stokes for their interest in this 
work. Part of this paper was presented in the
     \lq\lq International Meeting
      on Numerical Semigroups With Applications" (2016) at 
      Levico-Terme, Italy. We are deeply grateful to the referees for their comments, 
      suggestions and corrections that allowed to improve the early version of the paper.

\end{document}